\documentclass[10pt]{amsart}
\usepackage{amsmath, amsthm,amsxtra}
\usepackage[english]{babel}
\usepackage[T1]{fontenc}
\usepackage[latin1]{inputenc}
\usepackage{amssymb}
\usepackage{amscd}
\usepackage{latexsym}
\usepackage{graphicx}
\usepackage{mathrsfs} 
\usepackage[all,cmtip]{xy}

\usepackage{verbatim}

\newtheoremstyle{theorem}
  {12pt}          
  {12pt}  
  {\sl}  
  {\parindent}     
  {\bf}  
  {. }    
  { }    
  {}     
\theoremstyle{theorem}
\newtheorem{theorem}{Theorem}
\newtheorem{corollary}[theorem]{Corollary}
\newtheorem{remark}[theorem]{Remark}
\newtheorem{proposition}[theorem]{Proposition}
\newtheorem{lemma}[theorem]{Lemma}

\linethickness{0.5pt}

\newcommand{\ic}{\ensuremath{\mathcal{I}}}

\newcommand{\oc}{\ensuremath{\mathcal{O}}}

\newcommand{\fc}{\ensuremath{\mathcal{F}}}

\newcommand{\lc}{\ensuremath{\mathcal{L}}}

\newcommand{\mc}{\ensuremath{\mathcal{M}}}

\newcommand{\Pt}{\mathbb{P}^3}

\newcommand{\Ptw}{\mathbb{P}^2}

\newcommand{\Pn}{\mathbb{P}^n}

\newcommand{\bZ}{\mathbb{Z}}

\newcommand{\cG}{\gamma}
\newcommand{\oG}{\omega}

\newcommand{\bds}{\begin{displaystyle}}
\newcommand{\eds}{\end{displaystyle}}

\begin{document}
\title[On the cohomology of rank two vector bundles]{On the cohomology of rank two vector bundles on $\Ptw$ and a theorem of Chiantini and Valabrega.}

\author{Ph. Ellia}
\address{Dipartimento di Matematica e Informatica, Universit\`a degli Studi di Ferrara, Via Machiavelli 30, 44121 Ferrara, Italy.}
\email{phe@unife.it}


\begin{abstract} We show that a normalized rank two vector bundle, $E$, on $\Ptw$ splits if and only if $h^1(E(-1))=0$. Using this fact we give another proof of a theorem of Chiantini and Valabrega. Finally we describe the normalized bundles with $h^1(E(-1)) \leq 4$.
\end{abstract}

\date{\today}

\maketitle


\thispagestyle{empty}

\section{Introduction.}

We work over an algebraically closed field of characteristic zero. It follows from a famous theorem of Horrocks (\cite{OSS}) that a rank two vector bundle $E$ on $\Pn , n \geq 2$, splits if and only if $H^i_*(E):= \bigoplus _{k\in \bZ}H^i(E(k)) =0$, for $0<i<n$. This has been improved: as a consequence of another famous theorem by Evans-Griffith, under the same assumptions, $E$ splits if and only if $H^1_*(E)=0$ (see \cite{Ein}). Along these lines, on $\Pt$, there is a remarkable result:

\begin{theorem} \emph{(Chiantini-Valabrega \cite{ChiantiV})}\\
\label{T-CV}
Let $\fc$ be a rank two vector bundle on $\Pt$.\\
(1) If $c_1(\fc )=0$, then $\fc$ splits if and only if $h^1(\fc (-1))=0$.\\
(2) If $c_1(\fc )=-1$, then $\fc$ splits if and only if $h^1(\fc (-1))=0$ or $h^1(\fc )=0$ or $h^1(\fc (1))=0$.
\end{theorem}

It is natural to ask if there is a similar result on $\Ptw$ and indeed there is: let $E$ be a normalized (i.e. $-1 \leq c_1(E) \leq 0$) rank two vector bundle on $\Ptw$, then \emph{$E$ splits if and only if $h^1(E(-1))=0$}. Furthermore this is the best possible result. Indeed if $E = \Omega (1)$, then $h^1(E(m))=0, \forall m \neq -1$, but $E$ is indecomposable. Actually this result follows from a more general fact: \emph{with notations as above, we have $h^1(E(k)) \leq h^1(E(-1)), \forall k \in \bZ$} (see Theorem \ref{T-CV2}). The proof of Theorem \ref{T-CV2} is quite easy using standard vector bundles techniques. This statement has certainly been (unconsciously) known since a long time but, as far as I know, hasn't been put in evidence. That's a pity because it has some interesting consequences. For example we show how to recover Theorem \ref{T-CV} from it. (For another application see \cite{Tju}.)

In the last section, after some general considerations, we describe rank two vector bundles on $\Ptw$ with $h^1(E(-1)) \leq 4$.

\section{Variations on a theorem of Chiantini and Valabrega.}

Let us take some notations and recall some basic facts. 

If $F$ is a rank two vector bundle on $\Pn$, $n \geq 2$, then $c_1(F(m)) = c_1(F) + 2m$ and $c_2(F(m)) = c_1(F)m + c_2(F) + m^2$. A rank two vector bundle $E$ is \emph{normalized} if $-1 \leq c_1(E) \leq 0$. In this case we will denote by $c_1, c_2$ its Chern classes. 

\emph{In the sequel $E$ will \emph{always} denote a normalized rank two vector bundle with Chern classes $c_1, c_2$.}

The integer $r_E$ (or just $r$ if no confusion can arise) is defined as follows $r = min\{k \in \bZ \mid h^0(E(k)) \neq 0\}$. In other words $r$ is the least twist of $E$ having a section. Let $s \in H^0(E(r))$. If $s$ does not vanish, then $E \simeq \oc (-r) \oplus \oc (r+c_1)$. If $s$ vanishes, by minimality, its zero locus $(s)_0 = Z$, has codimension two and we have an exact sequence: $0 \to \oc \to E(r) \to \ic _Z(2r+c_1) \to 0$. The subscheme $Z$ is l.c.i. and $\deg (Z) = c_2(E(r))$.

The bundle $E$ is said to be \emph{stable} if $r > 0$. If $r \leq 0$ we will say that $E$ is \emph{not stable} (it can be semi-stable if $c_1=0$).

If $E$ is not stable and indecomposable, then $h^0(E(r)) = 1$, hence $Z$ is uniquely defined.

Finally we recall Riemann-Roch theorem:
If $F$ is a rank two vector bundle on $\Ptw$ with Chern classes $c_1, c_2$, then $$\chi (F) = 2 + \frac{c_1(c_1+3)}{2} -c_2.$$ In particular if $E$ is a normalized rank two vector bundle on $\Ptw$ with Chern classes $c_i$, then:
\begin{equation}
\label{eq.RR2}
\chi (E(k)) = \frac{c_1}{2}(c_1+2k+3) + (k+1)(k+2) -c_2.
\end{equation}
If $\fc$ is a rank two normalized vector bundle on $\Pt$ with Chern classes $c_i$, then:
\begin{equation}
\label{eq.RR3}\begin{split}
If\,\,c_1=0:\,\, \chi (\fc (k)) = -c_2(k+2) + \frac{1}{3}(k+1)(k+2)(k+3)\\
If\,\,c_1=-1:\,\, \chi (\fc (k)) = \frac{1}{6}(k+1)(k+2)(2k+3) -\frac{c_2}{2}(2k+3)
\end{split}
\end{equation}

Now we can prove the main result of this section:

\begin{theorem}
\label{T-CV2}
Let $E$ be a rank two normalized vector bundle on $\Ptw$. Then:\\
(1) $h^1(E(k)) \leq h^1(E(-1)), \forall k \in \bZ$.\\
(2) $E$ splits if and only if $h^1(E(-1))=0$.
\end{theorem}  

\begin{proof} (1) We may assume $E$ indecomposable. If $E$ is not stable we have an exact sequence: $0 \to \oc \to E(r) \to \ic _Z(2r+c_1) \to 0$, with $r \leq 0$ and $Z \subset \Ptw$ a non-empty zero-dimensional subscheme. Twisting by $\oc (-r-1)$ and taking cohomology we get: $h^1(E(-1)) = h^1(\ic _Z(r-1+c_1)$. Since $r-1+c_1 < 0$, $h^1(\ic _Z(r-1+c_1)) = h^0(\oc _Z) =: \deg (Z)$. Now for any $k$, the exact sequence above shows that $h^1(E(k)) \leq h^1(\ic _Z(k+r+c_1))$. Since $h^1(\ic _Z(m)) \leq h^0(\oc _Z), \forall m$ (consider $0 \to \ic _Z(m) \to \oc (m) \to \oc _Z \to 0$), we are done.

Now assume $E$ is stable. Let $L\subset \Ptw$ be a general line and consider the exact sequence $0 \to E(m-1) \to E(m) \to E_L(m) \to 0$. Since $E_L = \oc _L\oplus \oc _L(c_1)$ (Grauert-M\''ulich theorem, see \cite{OSS}), if $m \leq -1$, $h^0(E_L(m))=0$ and $h^1(E(m-1)) \leq h^1(E(m))$. It follows that $h^1(E(m)) \leq h^1(E(-1))$ if $m \leq -1$. If $m \geq 0$, by Serre duality $h^1(E(m)) = h^1(E^*(-m-3)) = h^1(E(-m-3-c_1))$ and again $h^1(E(m)) \leq h^1(E(-1))$.

(2) Of course (2) follows from (1) and Horrocks' theorem, but let us give a simpler argument. If $E$ is not stable arguing as in (1), we get $\deg Z =0$, hence $Z=\emptyset$ and $E$ splits. It remains to show that $h^1(E(-1)) > 0$ if $E$ is stable. By stability $\chi (E(-1)) = -h^1(E(-1))$. By Riemann-Roch, if $h^1(E(-1))=0$, we get $c_2=0$. Now $\chi (E) = 2$ if $c_1=0$ (resp. 1 if $c_1=-1$). It follows that $h^2(E) > 0$. But $h^2(E) = h^0(E^*(-3)) = h^0(E(-c_1-3))=0$, by stability. Hence $h^1(E(-1)) \neq 0$.
\end{proof}

\begin{remark} This is the best possible result in the sense that for any $m \neq -1$, there exists an indecomposable rank two vector bundle, $E$, with $h^1(E(m)) =0$: just take $E = \Omega (1)$.
\end{remark}

\begin{remark} Let's consider an unstable rank two vector bundle, $E$, with $c_1(E)=-1$. Arguing as above we see that $h^1(E)=0$ implies that $E$ splits. 

Assume now $h^1(E(1))=0$. We have $0 \to \oc \to E(r) \to \ic _Z(2r-1) \to 0$. Twisting by $\oc (-r+1)$ we get: $0 \to \oc (-r+1)\to E(1) \to \ic _Z(r) \to 0$ it follows that $h^1(\ic _Z(r))=0$. Now consider $0 \to \ic _Z(r) \to \oc (r) \to \oc _Z \to 0$. Since $r \leq 0$, the only possibility is $r=0$ and $\deg Z=1$. In conclusion, if $E$ doesn't split, we have: $0 \to \oc \to E \to \ic _P(-1) \to 0$, where $P$ is a point. Such bundles do exist.   
\end{remark}

\begin{remark}
\label{R-h1E>0 P2} One can show the following: let $E$ be a stable, rank two vector bundle on $\Ptw$, with $c_1(E)=-1$.

If $h^1(E)=0$ then there exists an exact sequence: $0 \to \oc \to E(1) \to \ic _Z(1) \to 0$, where $Z$ is a set of three non collinear points. We have $c_2(E) = 3$.

If $h^1(E(1))=0$ then there exists an exact sequence: $0 \to \oc \to E(2) \to \ic _Z(3) \to 0$, where $Z$ is a set of six points not lying on a conic. We have $c_2(E) = 4$.
\end{remark}

Let us recover Theorem \ref{T-CV}.

\begin{lemma} (1) Let $\fc$ be a stable, normalized, rank two vector bundle on $\Pt$. Then $h^1(\fc (-1)) \neq 0$.\\
(2) Moreover if $c_1(\fc )=-1$, we have $h^1(\fc ).h^1(\fc (1)) \neq 0$.
\end{lemma}

\begin{proof} (1) Let $H \subset \Pt$ be a general plane and consider the exact sequence $0 \to \fc (-2) \to \fc (-1) \to \fc _H(-1) \to 0$. Assume $h^1(\fc (-1)) = 0$. By Barth's restriction theorem (\cite{Ba}) $h^0(\fc _H(-1))=0$. It follows that $h^1(\fc (-2))=0$ and then $h^1(\fc (-m))=0, m \geq 1$. Now we have $h^2(\fc (-2)) = h^1(\fc (-c_1-2)) =0$. This implies $h^1(\fc _H(-1))=0$ and by Theorem \ref{T-CV2}, $\fc _H$ splits. This implies that $\fc$ also splits (see \cite{OSS}), a contradiction. Hence $h^1(\fc (-1)) \neq 0$.

(2) Assume $h^1(\fc )=0$. By stability we have $h^3(\fc ) = h^0(\fc (-3)) =0$. It follows that $\chi (\fc ) = h^2(\fc ) \geq 0$. By Riemann-Roch we get $1-3c_2/2 \geq 0$. This is impossible since $c_2 > 0$ and $c_2$ is even.

Assume $h^1(\fc (1))=0$. We have $h^3(\fc (1)) = h^0(\fc (-4))=0$. It follows that $\chi (\fc (1)) \geq 0$. By Riemann-Roch this yields: $5 -5c_2/2 \geq 0$. Since $c_2$ is even and $c_2>0$, it follows that $c_2 = 2$. Stable rank two vector bundles on $\Pt$ with $c_1=-1, c_2=2$ have been classified (\cite{Ha-Sols}) and they all have $h^1(\fc (1))=1$.
\end{proof}

\begin{lemma} Let $\fc$ be a non-stable, normalized, rank two vector bundle on $\Pt$. If $h^1(\fc (-1))=0$, then $\fc$ splits.
\end{lemma}

\begin{proof} Since $\fc$ is not stable we have an exact sequence: $0 \to \oc \to \fc (r) \to \ic _C(2r+c_1) \to 0\,\,\,(*)$, where $r \leq 0$ and where $C$ is either empty or a l.c.i. curve with $\oG _C(4-2r-c_1) \simeq \oc _C\,\,\,(**)$. Assume $h^1(\fc (-1))=0$ and $C$ non empty. Twisting by $\oc (-r-1)$ and taking cohomology, we get $h^1(\ic _C(r-1+c_1))=0$. Since $r-1+c_1 < 0$ this implies $h^0(\oc _C(r-1+c_1))=0$. It follows from $(**)$ that $h^0(\oG _C(-r+3))=0$. Now consider the exact sequence: $0 \to \ic _C(r-2+c_1) \to \ic _C(r-1+c_1) \to \ic _{C\cap H}(r-1+c_1) \to 0$, where $H$ is a general plane. If $h^2(\ic _C(r-2+c_1))=0$, then $h^1(\ic _{C\cap H}(r-1+c_1))=0$. Restricting $(*)$ to $H$ and twisting by $-r-1$, we get $h^1(\fc _H(-1))=0$. By Theorem \ref{T-CV2}, $\fc _H$ splits, hence $\fc$ also splits, which contradicts the minimality of the twist $r$ ($C$ should be empty). So $h^2(\ic _C(r-2+c_1)) = h^1(\oc _C(r-2+c_1)) \neq 0$. By Serre duality on $C$: $h^1(\oc _C(r-2+c_1)) = h^0(\oG _C(-r+2-c_1))\neq 0$. But this contradicts $h^0(\oG _C(-r+3))=0$. We conclude that $C$ is empty and that $\fc$ splits.
\end{proof}

\begin{lemma} Let $\fc$ be a non stable rank two vector bundle on $\Pt$, with Chern classes $c_1 =-1, c_2$. If $h^1(\fc )=0$ or $h^1(\fc (1)) =0$, then $\fc$ splits.
\end{lemma}

\begin{proof} Since $\fc$ is not stable we have an exact sequence: $0 \to \oc \to \fc (r) \to \ic _C(2r-1) \to 0$, with $r \leq 0$. Twisting by $\oc (m)$ and taking cohomology, we see that $h^0(\fc (m+r)) = h^0(\oc (m))$ as long as $m \leq -2r+1$ (since then $h^0(\ic _C(m+2r-1))=0$). Twisting by $\oc _H$, $H$ a general plane, we get $0 \to \oc _H \to \fc _H(r) \to \ic _{C\cap H}(2r-1) \to 0$. Arguing as above we get that $h^0(\fc _H(m+r))=h^0(\oc _H(m))$ if $m \leq -2r+1$. We conclude that the exact sequence $0 \to \fc (k-1) \to \fc (k) \to \fc _H(k) \to 0$ is exact on $H^0$ if $k \leq -r+1$. In particular we have $0 \to H^1(\fc (k-1)) \to H^1(\fc (k))$, if $k \leq -r+1$. If $h^1(\fc (t_0))=0$ with $t_0 \leq -r+1$, then $h^1(\fc (m))=0$ for $m \leq t_0$. So if $h^1(\fc ).h^1(\fc (1))=0$, then $h^1(\fc (-1))=0$. Since $h^2(\fc (-2))=h^1(\fc (-1))$, we get $h^1(\fc _H(-1))=0$. By Theorem \ref{T-CV2} we conclude that $\fc _H$ splits, hence $\fc$ also splits.
\end{proof}

Putting every thing together we get:

\begin{proposition}
\label{P-CV2CV} Theorem \ref{T-CV2} implies Theorem \ref{T-CV}.
\end{proposition}

\begin{remark} The original proof in \cite{ChiantiV} has been worked out in the framework of subcanonical space curves.
\end{remark}

Let us conclude this section with a last remark:

\begin{proposition} Let $E$ be an indecomposable rank two vector bundle on $\Ptw$. Then the module $H^1_*(E)$ is \emph{connected} (i.e. if $h^1(E(t))\neq 0$ and $h^1(E(m))\neq 0$ with $t < m$, then $h^1(E(k)) \neq 0$ for $t < k <m$).\\
According to Theorem \ref{T-CV2} this is equivalent to the following: (a) if $h^1(E(-t))=0$ for some $t \geq 2$, then $h^1(E(-m))=0, \forall m \geq t$, and (b) if $h^1(E(t))=0$ for some $t \geq 0$, then $h^1(E(m))=0, \forall m \geq t$.
\end{proposition}

\begin{proof} (1) First assume $E$ stable. Using the exact sequence $0 \to E(-t-1) \to E(-t) \to E_L(-t) \to 0$ ($L \subset \Ptw$ a general line) and the fact that $h^0(E_L(-t))=0$ if $t \geq 2$ (because $E_L \simeq \oc _L\oplus \oc _L(c_1)$, by stability), condition (a) follows immediately.

Now (b) follows from (a) by duality, indeed $h^1(E(t)) = h^1(E(-t-c_1-3))$ and $t+3+c_1 \geq 2$.

(2) Assume $E$ non stable. Then we have an exact sequence $0 \to \oc \to E(r) \to \ic _Z(2r+c_1) \to 0$, with $r \leq 0$, $Z \subset \Ptw$ zero-dimensional. If $h^1(E(t))=0, t \geq 0$, then $h^1(\ic _Z(2r+c_1+t))=0$. Since $Z$ is zero-dimensional we have $h^1(\ic _Z(k))=0, \forall k \geq 2r+c_1+t$, hence $h^1(E(m))=0, \forall m \geq t$. This proves (b). Now (a) follows by duality: by assumption $0= h^1(E(-t)) = h^1(E(t-c_1-3))$. Since $t \geq 2$, $t-c_1-3 \geq 0$, except if $c_1=0, t=2$ but this case cannot occur since $h^1(E(-1)) = h^1(E(-2)) \neq 0$ by Theorem \ref{T-CV2}. So if $c_1=0$, we may assume $t \geq 3$.
\end{proof}

\begin{remark} (i) This improves Castelnuovo-Mumford's lemma at least for the vanishing part.\\
(ii) It can be shown that the $H^1$-module of an indecomposable rank two vector bundle on $\Pt$ is connected, but the proof is much more difficult, see \cite{Bu}.
\end{remark}

\section{Rank two vector bundles on $\Ptw$ with $h^1(E(-1)) \leq 4$ .}

In this section we will investigate bundles with $h^1(E(-1))=:u$ small, say $u \leq 4$. Let us start with a useful remark:

\begin{remark}
\label{R-u} Assume $E$ indecomposable, $r$ as usual and consider $0 \to \oc \to E(r) \to \ic _Z(2r+c_1) \to 0$, where $Z \subset \Ptw$, is zero-dimensional. Let $0 \to \lc _1 \to \lc _0 \to \ic _Z \to 0$ be the minimal free resolution of $\ic _Z$. Then we can lift the morphism $\lc _0(2r+c_1) \to \ic _Z(2r+c_1)$ to a morphism $\lc _0(2r+c_1) \to E(r)$ and then get (after a twist) an exact sequence:
\begin{equation}
\label{eq:mfrE}
0 \to \lc _1(r+c_1) \to \oc (-r)\oplus \lc _0(r+c_1) \to E \to 0
\end{equation}
\label{eq:mfrEdual}
This gives the minimal free resolution of $H^0_*(E)$. Now by dualizing and taking into account that $E^* = E(-c_1)$ we get:
\begin{equation}
0 \to E \to \oc (r+c_1)\oplus \lc _0^*(-r) \to \lc _1^*(-r) \to 0
\end{equation}
Taking cohomology we get the beginning of the minimal free resolution of the $S:= k[x,y,z]$ module $H^1_*(E)$:
$$0 \to H^0_*(E) \to S(r+c_1)\oplus L_0^*(-r) \to L_1^*(-r) \to H^1_*(E) \to 0$$
Then combining with (\ref{eq:mfrE}) we get the whole resolution. By the way we notice a curious fact: $rk(S(r+c_1)\oplus L_0^*(-r)) = rk(L_1^*(-r))+2$. So for a finite length graded module $M$, to be the $H^1$-module of a rank two vector bundle on $\Ptw$, the number of relations among its generators must be the number of generators plus two. In fact this is not only necessary but also sufficient (see \cite{Rao} for details). 
\end{remark}

\begin{lemma}
\label{L-u}
Let $E$ be a normalized rank two vector bundle on $\Ptw$. Assume $E$ indecomposable, with $h^1(E(-1))=:u$. Let $r$ be the minimal twist of $E$ having a section. If $E$ is not stable, then $E(r)$ has a section vanishing on a zero-dimensional subscheme, $Z$, with $\deg (Z) =u$.
\end{lemma}

\begin{proof} We have an exact sequence $0 \to \oc \to E(r) \to \ic _Z(2r+c_1) \to 0$, with $r \leq 0$ since $E$ is not stable. Twisting by $\oc (-r+1)$ and taking cohomology we get: $h^1(E(-1)) = h^1(\ic _Z(r+c-1 - 1) = h^0(\oc _Z )$, because $r+c_1-1 < 0$ (notice that $-r-1 \geq -1$, hence $h^2(\oc (-r-1))=0$). It follows that $\deg (Z) = u$.
\end{proof}

\begin{remark} (1) In view of this lemma and on Remark \ref{R-u} if we know all the possible minimal free resolutions of $u$ points we get all possible resolutions of $H^0_*(E)$. Observe that the minimal free resolution of $H^0_*(E)$ determines the whole cohomology of $E$. Indeed if we know $h^0(E(k)), \forall k \in \bZ$, then by duality we know $h^2(E(k)), \forall k \in \bZ$. Knowing $h^0(E(k)), h^2(E(k))$, we get $h^1(E(k))$ by Riemann-Roch.\\
(2) If $E$ is non stable, indecomposable, then $h^0(E(r)) = 1$, hence $Z =(s)_0$ is uniquely defined. So we can define a map, $\cG$, from the set of non stable bundles with $h^1(E(-1))=u$ to $Hilb^u(\Ptw )$, by $\cG (E) = Z$.
\end{remark}

\begin{lemma}
\label{L-c2=u sta} Let $E$ be a  stable, normalized, rank two vector bundle on $\Ptw$. We have $u := h^1(E(-1)) = c_2$.
\end{lemma}

\begin{proof} Since $h^0(E(-1)) = 0 = h^2(E(-1)) = h^0(E(-c-1-2))$, we have $\chi (E(-1)) = -h^1(E(-1))$. By Riemann-Roch $\chi (E(-1)) = -c_2$ and the result follows.
\end{proof}

\begin{remark}
\label{R-class} At this point the classification, or better the description, of rank two vector bundles $E$ with $h^1(E(-1))=u$ can be split into two parts:\\
 (1) for non stable bundles: it is enough to determine all the minimal free resolutions of l.c.i., zero-dimensional subschemes of degree $u$. 

\noindent (2) classification of stables vector bundles of Chern classes $-1 \leq c_1 \leq 0$ and $c_2 = u$. In particular we want to know the least twist having a section.

Observe that the set of non stable bundles with $h^1(E(-1))=u$ is some kind of counterpart to the moduli space $\mc (c_1, c_2)$ ($c_2=u$) in the stable case.
\end{remark}

Let us start with non stable bundles. To make things manageable we will assume $u \leq 4$.

\begin{lemma}
\label{L-res u<5}
Let $Z \subset \Ptw$ be a closed subscheme of codimension two, with $\deg (Z) =u \leq 5$. There are ten possible minimal free resolutions for the ideal of $Z$, namely:\\
(a) $Z$ is contained in a line, in this case $Z$ is a complete intersection $(1,u)$\\
(b1) $u = 3$ and $Z$ is not contained in a line, in this case:\\ $0 \to 2.\oc (-3) \to 3.\oc (-2) \to \ic _Z\to 0$.\\
(b2) $u=4$, $h^0(\ic _Z(1))=0$, but $Z$ has a subscheme of length three contained in a line. In this case:
$0 \to \oc (-3)\oplus \oc (-4) \to 2.\oc (-2)\oplus \oc (-3) \to \ic _Z \to 0$\\
(b3) $u=4$ and $Z$ is a complete intersection $(2,2)$.\\
(b4) $u = 5$, $h^0(\ic _Z(1))=0$ but $Z$ has a subscheme of length 4 contained in a line. In this case:
$0 \to \oc (-3)\oplus \oc (-5) \to 2.\oc (-2)\oplus \oc (-4) \to \ic _Z \to 0$\\
(b5) $u=5$, $h^0(\ic _Z(2))=1$. In this case:\\
$0 \to  2.\oc (-4) \to 2.\oc (-3)\oplus \oc (-2) \to \ic _Z \to 0$\\
\end{lemma}

\begin{proof} Well known.
\end{proof}

As explained before this gives us all the possible resolutions (hence all the possible cohomologies) of non stable, indecomposable bundles with $h^1(E(-1)) \leq 4$. We need $u=5$ for the stable case:

\begin{proposition}
\label{P-stable u<5}
Let $E$ be a stable, normalized, rank two vector bundle on $\Ptw$, with $h^1(E(-1)) = u \leq 4$. As usual let $r$ denote the minimal twist of $E$ having a section. Then $r=1$ or $r = 2, c_1 = -1, u=4$. Moreover:\\
(1) If $c_1=0$ we have $u \geq 2$ and $E(1)$ has a section vanishing on a subscheme of degree $u+1$ which is not contained in a line.\\
(2) If $c_1 = -1$ and $r=1$, we have $u \geq 1$ and $E(1)$ has a section vanishing on a subscheme of length $u$. If $r = 2, u=4$, then $E(2)$ has a section vanishing on a degree 6 subscheme, $Z$, with $h^0(\ic _Z(2))=0$.
\end{proposition}

\begin{proof} In any case $h^2(E(1)) = h^0(E(-c_1-4))=0$ by stability. Since $\chi (E(1)) = 6-c_2$ if $c_1=0$ (resp. $4 - c_2$ if $c_1 = -1$) and since $c_2 = u$ (Lemma \ref{L-c2=u sta}), we get $\chi (E(1)) > 0$, except if $c_1=-1, u=4$. In this case we have $h^0(E(1)) = h^1(E(1))$. Assume $h^0(E(1))=0$. Since $\chi (E(2)) = 9-c_2$, we have $h^0(E(2)) > 0$ and an exact sequence $0 \to \oc \to E(2) \to \ic _Z(3) \to 0$, where $\deg (Z) = c_2(E(2)) = 6$. We have $h^0(E(1)) = 0 = h^0(\ic _Z(2))$. This proves the first claim.\\
(1) Assume $c_1=0$. We have $0 \to \oc \to E(1) \to \ic _Z(2) \to 0$. By stability $h^0(\ic _Z(1))=0$. In particular $deg Z = c_2+1 \geq 3$, i.e. $u = c_2 \geq 2$.\\
(2) Assume now $c_1 = -1$. Since $c_2(E(1)) = c_2 = u$, if $r=1$, we have $0 \to \oc \to E(1) \to \ic _Z(1) \to 0$, with $Z$ of degree $u$. If $u=4$ and $h^0(E(1))=0$, a section of $E(2)$ vanishes along $Z$ of degree 6 with $h^0(E(1)) = 0 = h^0(\ic _Z(2))$.
\end{proof}

\begin{remark} (1) By Serre's construction for any $k \leq 2$ and any locally complete intersection, zero-dimensional subscheme $Z\subset \Ptw$ there exists a rank two vector bundle, $F$, with $c_1(F) = k$, and an exact sequence $0 \to \oc \to F \to \ic _Z(k) \to 0$. If $k \leq 0$ this is the least twist of $F$ having a section and $F$ is not stable. In particular all the bundles we have considered in Proposition \ref{P-stable u<5} do really exist !\\
(2) We have the list of all possible resolutions for $H^0_*(E)$, where $E$ is a normalized bundle with $h^1(E(-1)) \leq 4$. Indeed the only case not covered by Lemma is when $r=2$, but a subscheme, $Z$, of degree 6, not on a conic has a resolution like: $0 \to 3.\oc (-4) \to 4.\oc (-3) \to \ic _Z \to 0$.

(2) We observe that if $u=1$ we always have that $E(r)$ has a section vanishing at one point. More precisely:
\end{remark}

\begin{corollary}
Let $E$ be a normalized, indecomposable, rank two vector bundle on $\Ptw$. Let $r$ denote the minimal twist of $E$ having a section.\\
(1) The following are equivalent:\\
(i) $h^1(E(m)) \leq 1, \forall m \in \bZ$\\
(ii) $h^1(E(-1)) = 1$\\
(iii) $E(r)$ has a section vanishing at one point\\
(iv) there is an exact sequence:
$$0 \to \oc (-b-1) \to \oc (-a) \oplus 2.\oc (-b) \to E \to 0$$
with $a \leq b$ (in particular $a=r, b=-r-c_1+1$).\\
(2) A bundle like in (1) is stable if and only if $a = b$, if and only if $E = \Omega (1)$.
\end{corollary}

\begin{proof} (i) $\Leftrightarrow$ (ii), since $E$ is indecomposable, this follows from Theorem \ref{T-CV2}.\\
(ii) $\Rightarrow$ (iii): If $E$ is non stable this follows from Lemma \ref{L-u}. If $E$ is stable this follows from Proposition \ref{P-stable u<5}. More precisely we have $c_2 = 1$ (Lemma \ref{L-c2=u sta}) and $r=1$, $c_1 =-1$ (Proposition \ref{P-stable u<5}).\\
(iii) $\Rightarrow$ (iv): This follows from Remark \ref{R-u}.\\
(iv) $\Rightarrow$ (i): Since $a \leq b$, $r=a$, hence a section of $E(a)$ will vanish in codimension two. Since $c_1 = -r-b+1$, we get $b=-r-c_1+1$, $c_1(E(a)) = a-b+1$. We get a commutative diagram:
$$\begin{array}{ccccccccc}
 & & & &0 & &0 & & \\
 & & & &\downarrow & &\downarrow & & \\
 & & & &\oc &= &\oc & & \\
 & & & &\downarrow & &\downarrow & & \\
0 & \to & \oc (a-b-1) &\to & \oc \oplus 2.\oc (-b+a)&\to & E(a)& \to & 0\\
 & &|| & &\downarrow & &\downarrow & & \\
 0 & \to & \oc (a-b-1) &\to &  2.\oc (-b+a)&\to & \ic _Z(a-b+1)& \to & 0\\
  & & & &\downarrow & &\downarrow & & \\
 & & & &0 & &0 & & \\
\end{array}$$
So we get $0 \to \oc (-2) \to 2.\oc (-1) \to \ic _Z \to 0$ and we conclude that $Z$ is a point $p$. Since $h^1(\ic _p(m))$ is $0$ if $m\geq 0$ and $1$ if $m<0$, we conclude that $h^1(E(k)) \leq 1, \forall k \in \bZ$.\\
(2) We have already seen ((ii) $\Rightarrow$ (iii)) that $E$ is stable if and only if $c_2 = r =1$, $c_1 =-1$. Hence we have $0 \to \oc (-1) \to 3.\oc \to E(1) \to 0$. It follows that $E(1) = T(-1) = \Omega (2)$. On the other hand if $r = a = b = -r-c_1 +1$, then $c_1 = -1$ and $r = 1$, in particular $E$ is stable.  
\end{proof}

\begin{remark} This result is known in the context of logarithmic bundles, see \cite{DS}, \cite{Mar-Vallès}
\end{remark}

In the same vein we have:

\begin{corollary}
Let $E$ be a normalized, indecomposable, rank two vector bundle on $\Ptw$. Let $r$ denote the minimal twist of $E$ having a section. The following are equivalent:\\
(i) $h^1(E(-1)) = 2$\\
(ii) $E(r)$ has a section vanishing along a subscheme of degree two, or $E$ is stable with $c_1 = 0, c_2 = 2, r =1$ and $E(1)$ has a section vanishing along a subscheme of degree three not contained in a line.\\
(iii) there is an exact sequence:
$$0 \to \oc (-b-2) \to \oc (-b-1) \oplus \oc (-b) \oplus \oc (-a) \to E \to 0$$
with $a \leq b$ (in particular $a=r, b=-r-c_1+1$), or:\\
$0 \to 2.\oc (-2) \to 4.\oc (-1) \to E \to 0$.
\end{corollary}

\begin{proof} It is similar to the previous one, so we omit it.
\end{proof}




\end{document}